\documentclass[10pt]{article}

\usepackage{lmodern}
\usepackage[T1]{fontenc}
\usepackage{amsmath,amsthm}
\usepackage{mathtools}
\usepackage{tikz}
\usepackage{titlefoot}
\usepackage[small]{titlesec}
\usepackage{amssymb}

\renewcommand{\P}{\mathbb{P}}
\newcommand{\ds}{\displaystyle}

\newcommand{\f}[2]{\ds\frac{{#1}}{{#2}}}

\newcommand{\st}{\::\:}

\usepackage{color}
\definecolor{lightgray}{rgb}{0.8, 0.8, 0.8}
\definecolor{darkgray}{rgb}{0.7, 0.7, 0.7}
\definecolor{darkblue}{rgb}{0, 0, .4}

\newcommand{\fnmatrix}[2]{\mbox{\begin{footnotesize}$\begin{array}{#1}#2\end{array}$\end{footnotesize}}}
\newcommand{\eval}[2][\right]{\relax\ifx#1\right\relax \left.\fi#2#1\rvert}

\makeatletter

\let\c@table\c@figure
\makeatother 

\usepackage[bookmarks]{hyperref}
\hypersetup{
        colorlinks=true,
        linkcolor=darkblue,
        anchorcolor=darkblue,
        citecolor=darkblue,
        urlcolor=darkblue,
        pdfpagemode=UseThumbs,
        pdftitle={Rearrangement Conjecture},
        pdfsubject={Combinatorics},
        pdfauthor={Pantone, and Vatter},
}

\newcounter{todocounter}

\theoremstyle{plain}
\newtheorem{theorem}{Theorem}[section]
\newtheorem{proposition}[theorem]{Proposition}

\newtheorem{conjecture}[theorem]{Conjecture}

\makeatletter
\newfont{\footsc}{cmcsc10 at 8truept}
\newfont{\footbf}{cmbx10 at 8truept}
\newfont{\footrm}{cmr10 at 10truept}
\renewenvironment{abstract}%
                {
                  \begin{list}{}%
                     {\setlength{\rightmargin}{1in}%
                      \setlength{\leftmargin}{1in}}%
                   \item[]\ignorespaces\begin{small}}%
                 {\end{small}\unskip\end{list}}

\datefoot{\today}

\newpagestyle{main}[\small]{
        \headrule
        \sethead[\usepage][][]
        {\sc Rearrangement Conjecture}{}{\usepage}}

\author{Jay Pantone\qquad\qquad Vincent Vatter\thanks{Vatter's research was sponsored by the National Security Agency under Grant Number H98230-12-1-0207 and the National Science Foundation under Grant Number DMS-1301692.  The United States Government is authorized to reproduce and distribute reprints not-withstanding any copyright notation herein.}\\[-0.25ex]
\small Department of Mathematics\\[-0.5ex]
\small University of Florida\\[-0.5ex]
\small Gainesville, Florida, USA\\[-1.5ex]
}

\title{\sc On the Rearrangement Conjecture for Generalized Factor Order Over $\mathbb{P}$}

\titleformat{\section}
        {\large\sc}
        {\thesection.}{1em}{}

\date{}

\begin{document}

\maketitle

\begin{abstract}
The Rearrangement Conjecture states that if two words over $\mathbb{P}$ are Wilf-equivalent in the factor order on $\mathbb{P}^\ast$ then they are rearrangements of each other. We introduce the notion of strong Wilf-equivalence and prove that if two words over $\mathbb{P}$ are strongly Wilf-equivalent then they are rearrangements of each other. We further conjecture that Wilf-equivalence implies strong Wilf-equivalence.
\end{abstract}

\section{Introduction}

For ordinary words (finite sequences) $u$ and $w$, we say that $u$ is a \emph{factor} of $w$ if $w=w^{(1)}uw^{(2)}$ for possibly empty words $w^{(1)}$ and $w^{(2)}$. We are concerned with the \emph{generalized factor order}, which extends the factor order to words over an arbitrary poset $P$. Given words $u,v\in P^\ast$, we say that $v$ \emph{dominates} $u$ if they have the same length and $v_i\ge_P u_i$ for all $i$. We say that $u$ is a \emph{factor} of $w$ if $w=w^{(1)}vw^{(2)}$ for a word $v$ which dominates $u$, and in this case we write $w\ge_\text{gfo} u$. For example, when $P = \mathbb{P}$, the positive integers, we have $1423314 \ge_{\text{gfo}} 3123$ because $4233$ dominates $3123$. This paper is solely concerned with the case where $P=\mathbb{P}$.

The primary generating function we are interested in enumerates the set of all words in $\mathbb{P}^\ast$ (which could also be thought of as compositions) according to their length $|w|$, the sum of their entries $\|w\|$, and the number of factors dominating $u$ they contain:
\[
A_u(x,y,z) = \ds\sum_{w \in \P^\ast} x^{|w|}y^{\|w\|}z^{\text{\# of factors dominating $u$}}.
\]

The generalized factor order was introduced by Kitaev, Liese, Remmel, and Sagan~\cite{kitaev:rationality-irr:}, who defined the words $u,v \in \mathbb{P}^\ast$ to be \emph{Wilf-equivalent} if $A_u(x,y,0) = A_v(x,y,0)$. They also made what has become known as the Rearrangement Conjecture. To state this conjecture, we need another definition. The words $u$ and $v$ of equal length are said to be \emph{rearrangements} if $u$ and $v$ have the same multiset of values, i.e., if there exists a permutation $\pi$ of the set $\{1,2,\dots,|u|\}$ such that $v = u_{\pi(1)}u_{\pi(2)}\cdots u_{\pi(|u|)}$.

\newtheorem*{rearrangement-conj}{The Rearrangement Conjecture}
\begin{rearrangement-conj}
If two words in $\P^\ast$ are Wilf-equivalent then they are rearrangements of each other.
\end{rearrangement-conj}

Note that the converse to the Rearrangement Conjecture is false. Following the methods of~\cite{kitaev:rationality-irr:}, we can construct automata which recognize words avoiding a given factor. From this we find that
	\[A_{122}(x,y,0) = \f{1 - 2y + (1+x)y^2 - xy^3 + x^2y^4}{1 - (2+x)y + (1+2x)y^2 - (x+x^2)y^3 + x^2y^4},\]
while
	\[A_{212}(x,y,0) = \f{1-2y+(1+x)y^2-(x-x^2)y^3+x^3y^5}{(1-y+x^2y^3)(1-(1+x)y+xy^2-x^2y^3)}.\]
In particular, $[x^4y^7]A_{122}(x,y,0) = 13$, while $[x^4y^7]A_{212}(x,y,0) = 12$.

In \cite{kitaev:rationality-irr:}, the authors also introduced a stronger notion of Wilf-equivalence. Given words $u,w\in\P^\ast$ we define $\text{Em}(u,w)$ to consist of the set of indices of letters in $w$ which begin a factor dominating $u$. The words $u$ and $v$ are said to be \emph{super-strongly Wilf-equivalent} if there is a bijection $f\st\P^\ast\rightarrow\P^\ast$ such that $|f(w)|=|w|$, $\|f(w)\|=\|w\|$, and $\text{Em}(u,f(w))=\text{Em}(v,w)$ for all $w\in\P^\ast$. (Kitaev, Liese, Remmel, and Sagan had called this property ``strong Wilf-equivalence'', but we use that term to mean something different, in keeping with related literature.) This is such a stringent condition that while the words $2143$ and $3412$ are trivially Wilf-equivalent (because one is the reverse of the other), they are not super-strongly Wilf-equivalent (see \cite{kitaev:rationality-irr:}).

We focus on a condition which lies between these two; we say that $u$ and $v$ are \emph{strongly Wilf-equivalent} if $A_u(x,y,z)=A_v(x,y,z)$. Note that super-strong Wilf-equivalence implies strong Wilf-equivalence, which in turn implies Wilf-equivalence. Our main result is the following.


\begin{theorem}
\label{thm-rearrangement}
If two words in $\P^\ast$ are strongly Wilf-equivalent then they are rearrangements of each other.
\end{theorem}

Moreover, we conjecture that the following is true, which would imply the full Rearrangement Conjecture.

\begin{conjecture}
\label{conj-wilf}
If two words in $\P^\ast$ are Wilf-equivalent then they are strongly Wilf-equivalent.
\end{conjecture}

We have tested Conjecture~\ref{conj-wilf} by finding all Wilf-equivalence classes for factors of weight up to 11 contained in words of weight up to 20 and verifying that each such class is also a strong Wilf-equivalence class (again, checking words of weight up to 20).

It is interesting that a related conjecture has recently been made for consecutive permutation patterns. Given a permutation $\beta$ of length $k$ and another permutation $\pi$ of length $n$, $\pi$ contains a consecutive occurrence of $\beta$ if there is an index $i$ such that the subsequence $\pi(i)$, $\pi(i+1)$, $\dots$, $\pi(i+k-1)$ is in the same relative order as $\beta$. The analogue of the $A_u$ generating function in this context is then
\[
	A_\beta(x,z)= \; \sum_{\mathclap{\text{permutations $\pi$}}} \; x^{|\pi|}z^{\text{\# of occurrences of $\beta$ in $\pi$}}.
\]
The permutations $\beta$ and $\gamma$ are said to be \emph{c-Wilf-equivalent} (the ``c'' is to denote that this is for consecutive pattern containment) if $A_\beta(x,0)=A_\gamma(x,0)$, and \emph{strongly c-Wilf-equivalent} if $A_\beta(x,z)=A_\gamma(x,z)$. Nakamura~\cite[Conjecture 5.6]{nakamura:computational-a:} has conjectured that c-Wilf-equivalence implies strong c-Wilf-equivalence.

\section{The Cluster Method}

The easiest way to compute the generating function $A_u(x,y,z)$ for a given word $u$ is probably to construct an automaton, as detailed in \cite{kitaev:rationality-irr:}. However, to prove Theorem~\ref{thm-rearrangement}, we instead use the \emph{cluster method}. This method is originally due to Goulden and Jackson~\cite{goulden:combinatorial-e:}. The authors owe their knowledge of the method to the recent work of Elizalde and Noy~\cite{elizalde:clusters-genera:} and Elizalde~\cite{elizalde:the-most-and-th:} on consecutive patterns in permutations.

Given a word $u\in\P^\ast$ of length $k$ (which will be the forbidden factor), an \emph{$m$-cluster} of $u$ is a word $c\in\P^\ast$ consisting of $m$ overlapping factors of length $k$ containing $u$, which are marked. The overlapping condition requires that when labeled from left to right, each pair of consecutive factors must share at least one entry. Additionally, the first and last letters of $c$ must be contained in marked factors.

Note that an $m$- and $m'$-cluster can share the same underlying word. For example, Figure~\ref{figure:cluster-example} shows a $2$- and a $3$-cluster of $u=3123$ defined on the same word in $\P^\ast$.

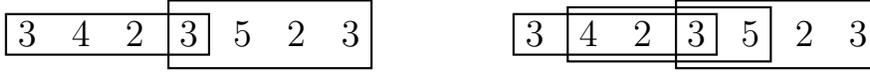
\begin{figure}
\begin{center}
	\hfill
	\begin{tikzpicture}[scale=.9]
		\node at (0,0) {\Large $3$ \; $4$ \; $2$ \; $3$ \; $5$ \; $2$ \; $3$};
		\draw[thick] (-2.7,-.3) rectangle (.3,.3);
		\draw[thick] (-.3,-.5) rectangle (2.7,.5);
	\end{tikzpicture}
	\hfill
	\begin{tikzpicture}[scale=.9]
		\node at (0,0) {\Large $3$ \; $4$ \; $2$ \; $3$ \; $5$ \; $2$ \; $3$};
		\draw[thick] (-2.7,-.3) rectangle (.3,.3);
		\draw[thick] (-1.9,-.4) rectangle (1.1,.4);
		\draw[thick] (-.3,-.5) rectangle (2.7,.5);
	\end{tikzpicture}
	\hfill\ 
	
\end{center}
	\caption{A 2-cluster and a 3-cluster of $3123$ which share the same underlying element of $\P^\ast$.}
	\label{figure:cluster-example}

\end{figure}

The cluster generating function for $u$ is defined as
\[
C_u(x,y,z)
=
\sum_{m\ge 1} z^m\;
\sum_{\mathclap{\substack{\text{$m$-clusters}\\\text{$c$ of $u$}}}} x^{|c|}y^{\|c\|}.
\]
Now fix a forbidden factor $u$. We may view an arbitrary word $w$ (which may or may not contain $u$ as a factor) as a sequence of letters and clusters of $u$. The generating function for an arbitrary letter of $\P$ is
\[
\frac{xy}{1-y}=xy+xy^2+\cdots,
\]
while the generating function for clusters is $C_u(x,y,z)$. We first claim that
\begin{equation}
\label{eqn-cluster-minus-1}
	A_u(x,y,0)=\frac{1}{1-\frac{xy}{1-y}-C_u(x,y,-1)}.
\tag{$\dagger$}
\end{equation}
To see why \eqref{eqn-cluster-minus-1} is true, consider the expansion of the right-hand side, which includes terms for every word $w\in\P^\ast$. If the word $w$ has $s$ factors which dominate $u$, then each such factor may or may not be chosen to participate in a cluster. Thus $w$ corresponds to $2^s$ terms in the expansion of the right-hand side of \eqref{eqn-cluster-minus-1}, half with positive signs and half with negative signs. Therefore what remains after cancellation is indeed $A_u(x,y,0)$. In particular, if $u$ and $v$ are Wilf-equivalent, then $C_u(x,y,-1)=C_v(x,y,-1)$.

To investigate strong Wilf-equivalence, we need to use what Wilf~\cite[Section 4.2]{wilf:generatingfunct:} calls the Sieve Method. Consider the generating function
\[
	\frac{1}{1-\frac{xy}{1-y}-C_u(x,y,z)}.
\]
If the word $w$ contains $s$ factors which dominate $u$ then the same analysis as above shows that the contribution of $w$ to the expansion of the above generating function is
\[
	x^{|w|}y^{\|w\|}\left({s\choose 0}z^0+{s\choose 1}z^1+\cdots+{s\choose s}z^s\right)
	=
	x^{|w|}y^{\|w\|}(z+1)^s.
\]
We therefore deduce by substitution that
\[
	A_u(x,y,z)=\frac{1}{1-\frac{xy}{1-y}-C_u(x,y,z-1)}.
\]

Thus the words $u$ and $v$ are strongly Wilf-equivalent if and only if $C_u(x,y,z) = C_v(x,y,z)$. In light of this, we may view Conjecture~\ref{conj-wilf} as stating that if $C_u(x,y,-1)=C_v(x,y,-1)$ then $C_u(x,y,z)=C_u(x,y,z)$.

In the context of the factor order on $\P^\ast$, increasing the entries of an $m$-cluster still leaves an $m$-cluster. Therefore we say that the $m$-cluster $c$ for $u$ is \emph{minimal} if none of its entries can be decreased without destroying a marked factor. (Note that the clusters in Figure~\ref{figure:cluster-example} are \emph{not} minimal.) We define the generating function for these clusters to be
\[
M_u(x,y,z)
=
\sum_{m\ge 1} z^m\;
\sum_{\mathclap{\substack{\text{minimal}\\ \text{$m$-clusters}\\\text{$c$ of $u$}}}} x^{|c|}y^{\|c\|},
\]
and it is easy to see that
\[
C_u(x,y,z)=M_u\left(\frac{x}{1-y},y,z\right).
\]
Therefore the words $u$ and $v$ are Wilf-equivalent if and only if $M_u(x,y,-1)=M_v(x,y,-1)$ and they are strongly Wilf-equivalent if and only if $M_u(x,y,z)=M_u(x,y,z)$. Note that $|u|$ is the smallest exponent of $x$ in $M_u(x,y,-1)$, so if two words are Wilf-equivalent they must have the same length. Similarly, $\|u\|$ is the smallest exponent of $y$ in $M_u(x,y,-1)$, so if $u$ and $v$ are Wilf-equivalent, $\|u\|=\|v\|$.

We can now outline our proof of Theorem~\ref{thm-rearrangement}. By sorting the entries of a word $u\in\P^\ast$ we obtain a partition $\lambda=(\lambda_1\ge\cdots\ge\lambda_k)$ of $\|u\|$. We prove our theorem by showing that one can compute $\lambda$ by examining the coefficients of $M_u(x,y,z)$. This proves the theorem because it means that if the words $u$ and $v$ are strongly Wilf-equivalent, then they are rearrangements of the same partition. In the next Section, we illustrate our proof by presenting the special case of a forbidden factor of length $4$. The proof of the general case is presented in the section after that. We conclude by showing how the cluster method can be applied to derive short proofs of a result and two conjectures of Kitaev, Liese, Remmel, and Sagan~\cite{kitaev:rationality-irr:}.

\section{The Case $k=4$}

We begin by establishing terminology which will be used in the proof of the general case. A minimal $m$-cluster of $u$ is built by aligning $m$ overlapping (but not necessarily mutually overlapping) copies of $u$ in an array and then taking the maximum of each column. For example, below are all minimal $2$-clusters of a word $u$ of length $k=4$, where for a set $I\subseteq\{1,2,3,4\}$ we use the notation $u_I=\max\{u_i\st i\in I\}$.

\begin{center}
\begin{tabular}{ccc}

\fnmatrix{ccccc}{
u_1&u_2&u_3&u_4\\
&u_1&u_2&u_3&u_4\\
\hline
u_1&u_{1,2}&u_{2,3}&u_{3,4}&u_{4}
}

&

\fnmatrix{cccccc}{
u_1&u_2&u_3&u_4\\
&&u_1&u_2&u_3&u_4\\
\hline
u_1&u_2&u_{1,3}&u_{2,4}&u_{3}&u_4
}

&

\fnmatrix{ccccccc}{
u_1&u_2&u_3&u_4\\
&&&u_1&u_2&u_3&u_4\\
\hline
u_1&u_2&u_3&u_{1,4}&u_2&u_3&u_4
}

\end{tabular}
\end{center}

\noindent Using bracket notation for coefficient extraction, we see that for a word $u$ of length $4$,
\[
[z^2]M_u
=
x^5 y^{u_1+u_{1,2}+u_{2,3}+u_{3,4}+u_4}
+
x^6 y^{u_1+u_2+u_{1,3}+u_{2,4}+u_3+u_4}
+
x^7 y^{u_1+u_2+u_3+u_{1,4}+u_2+u_3+u_4}.
\]
Of course we don't know what the $u_i$ are, but we can examine the terms of $M_u$ to determine these sums.

There will typically be more than one minimal $m$-cluster of each length for $m\ge 3$. We summarize the information on $2$-clusters in the following chart, where the number corresponding to length $m$ and entry $u_I$ represents the number of occurrences of $u_I$ among all $2$-clusters of length $m$. 
\[
\fnmatrix{c|cccc|cccccc}{
\mbox{length}&u_1&u_2&u_3&u_4&u_{1,2}&u_{1,3}&u_{1,4}&u_{2,3}&u_{2,4}&u_{3,4}\\\hline
5&1&&&1&1&&&1&&1\\
6&1&1&1&1&&1&&&1&\\
7&1&2&2&1&&&1
}
\]
While we selected $m=2$ in order to explain our approach, this data does not lead to any conclusions. Instead, we need to consider the $m=3$ and $m=4$ cases, for which we only present the charts. For $m=3$, we have the following.
\[
\fnmatrix{c|cccc|cccccc|cccc|c}{
\mbox{length}&u_1&u_2&u_3&u_4&u_{1,2}&u_{1,3}&u_{1,4}&u_{2,3}&u_{2,4}&u_{3,4}&u_{1,2,3}&u_{1,2,4}&u_{1,3,4}&u_{2,3,4}&u_{1,2,3,4}\\\hline
6&1&&&1&1&&&&&1&1&&&1&\\
7&2&1&1&2&1&1&&2&1&1&&1&1&&\\
8&3&3&3&3&2&2&2&2&2&2&&&&&\\
9&2&4&4&2&&2&2&&2&&&&&&\\
10&1&3&3&1&&&2&&&&&&&&
}
\]
Because every minimal cluster has precisely one letter equal to $u_1$ (its first), we can see from the above table that there are nine $3$-clusters of a pattern of length $4$. The data for $m=4$ is displayed below.
\[
\fnmatrix{c|cccc|cccccc|cccc|c}{
\mbox{length}&u_1&u_2&u_3&u_4&u_{1,2}&u_{1,3}&u_{1,4}&u_{2,3}&u_{2,4}&u_{3,4}&u_{1,2,3}&u_{1,2,4}&u_{1,3,4}&u_{2,3,4}&u_{1,2,3,4}\\\hline
7&1&&&1&1&&&&&1&1&&&1&1\\
8&3&1&1&3&2&1&&2&1&2&2&2&2&2&\\
9&6&4&4&6&5&4&3&5&4&5&2&2&2&2&\\
10&7&9&9&7&4&7&6&6&7&4&&2&2&&\\
11&6&12&12&6&3&6&9&3&6&3&&&&&\\
12&3&9&9&3&&3&6&&3&&&&&&\\
13&1&4&4&1&&&3&&&&&&&&
}
\]

We notice from the tables above that the shortest minimal $3$-cluster and the shortest minimal $4$-cluster, which have lengths $6$ and $7$ respectively, are almost identical except for the presence of $u_{1,2,3,4}$ in the (exponent of $y$ corresponding to the relevant) $4$-cluster. In terms of generating functions, this means that
\[
\eval{\frac{d}{dy}\left(\left([x^7z^4]-[x^6z^3]\right)M_u\right)}_{y=1}
=
u_{1,2,3,4}.
\]
Of course, $u_{1,2,3,4}$ is the largest entry of $u$, so we have just computed $\lambda_1$ from $M_u$. Comparing the length $8$, $m=4$ data with the lengths $6$ and $7$, $m=3$ data, we see that
\[
\eval{\frac{d}{dy}\left(\left([x^8z^4]-[x^7z^3]-[x^6z^3]\right)M_u\right)}_{y=1}
=
u_{1,2,3}+u_{1,2,4}+u_{1,3,4}+u_{2,3,4}.
\]
Of these four terms, three are equal to the greatest entry of $u$, while one is equal to the second greatest entry. Therefore
\[
\eval{\frac{d}{dy}\left(\left([x^8z^4]-[x^7z^3]-[x^6z^3]\right)M_u\right)}_{y=1}
=
3\lambda_1+\lambda_2.
\]
As we have previously determined $\lambda_1$, this allows us to compute $\lambda_2$. Next we compare the length $9$, $m=4$ data to the lengths $6$, $7$, and $8$, $m=3$ data to see that
\begin{eqnarray*}
\lefteqn{
	\eval{\frac{d}{dy}\left(\left([x^9z^4]-[x^8z^3]-[x^7z^3]-[x^6z^3]\right)M_u\right)}_{y=1}=
}
\\&&
u_{1,2}+u_{1,3}+u_{1,4}+u_{2,3}+u_{2,4}+u_{3,4}+
u_{1,2,3}+u_{1,2,4}+u_{1,3,4}+u_{2,3,4}.
\end{eqnarray*}
We then see that
\[
u_{1,2}+u_{1,3}+u_{1,4}+u_{2,3}+u_{2,4}+u_{3,4}
=
3\lambda_1+2\lambda_2+\lambda_3,
\]
and thus, by our observation above about $u_{1,2,3}+u_{1,2,4}+u_{1,3,4}+u_{2,3,4}$,
\[
\eval{\frac{d}{dy}\left(\left([x^9z^4]-[x^8z^3]-[x^7z^3]-[x^6z^3]\right)M_u\right)}_{y=1}
=
6\lambda_1+3\lambda_2+\lambda_3,
\]
enabling us to compute $\lambda_3$. It remains only to compute $\lambda_4$, but given that we know $\lambda_1$, $\lambda_2$, and $\lambda_3$, we can compute $\lambda_4$ by looking at the smallest exponent of $y$ in $M_u$, which is equal to $\|u\|=\lambda_1+\lambda_2+\lambda_3+\lambda_4$.

\section{The General Case}

Now suppose that the word $u$ has arbitrary length $k$. We aim to show that for every $1\le i\le k-1$, the quantity
\begin{equation}
\label{main-eq}
	\eval{\frac{d}{dy}\left(\left([x^{2k+i-2}z^k]-[x^{2k+i-3}z^{k-1}]-\cdots-[x^{2k-2}z^{k-1}]\right)M_u\right)}_{y=1}
\tag{$\ddagger$}
\end{equation}
is a linear combination (with strictly positive coefficients) of $\lambda_1,\dots,\lambda_i$. This will allow us to compute $\lambda_1,\dots,\lambda_{k-1}$. We then compute $\lambda_k$ by examining the smallest exponent of $y$ in $M_u$ which is equal to $\|u\|=\lambda_1+\cdots+\lambda_k$. We must first refine our terminology.

As we began the previous section, note that minimal $m$-clusters of $u$ are obtained by aligning $m$ overlapping copies of $u$ in an array and then taking the maximum of each column. For example, below we show a $5$-cluster of a word $u$ of length $5$.
\[
	\fnmatrix{ccccccccccc}{
	u_1&u_2&u_3&u_4&u_5\\
	&u_1&u_2&u_3&u_4&u_5\\
	&&&u_1&u_2&u_3&u_4&u_5\\
	&&&&u_1&u_2&u_3&u_4&u_5\\
	&&&&&&u_1&u_2&u_3&u_4&u_5\\
	\hline
	u_1&u_{1,2}&u_{2,3}&u_{1,3,4}&u_{1,2,4,5}&u_{2,3,5}&u_{1,3,4}&u_{2,4,5}&u_{3,5}&u_4&u_5
	}
\]
We refer to the word below the line as a \emph{symbolic $m$-cluster}, owing to the fact that we think about it throughout our proof as a word over the letters $u_I$ for $I\subseteq\{1,\dots,k\}$. We call the array above the line which produced the symbolic cluster an \emph{$m$-pre-cluster}. Thus every entry of a symbolic cluster comes from a column of its associated pre-cluster. Note that the presence of a given column in a pre-cluster uniquely determines the contents and relative position of the rows which have nonempty entries in that column.

In particular, in the expansion of \eqref{main-eq}, each term corresponds to a column of a $k$- or $(k-1)$-pre-cluster. We  say that the \emph{height} of a column of a pre-cluster is the number of nonempty entries it contains. A \emph{top column} of a pre-cluster is one that includes a nonempty entry of the top row. Similarly, a \emph{bottom column} is one that touches the bottom row. (It is possible for a column to be both top and bottom.) Finally, a \emph{middle column} is one that touches neither the top nor bottom row of the pre-cluster. 

We prove our claim about \eqref{main-eq}, and thus Theorem~\ref{thm-rearrangement}, with a series of five propositions.

\begin{proposition}
\label{prop-cancel-1}
In the expansion of \eqref{main-eq}, none of the remaining terms after cancellation correspond to top or bottom columns of height $k-2$ or less.
\end{proposition}
\begin{proof}
We prove the proposition by constructing a bijection $\Phi$ which maps a $k$-pre-cluster of length $2k+i-2$ with a single marked non-middle column of height at most $k-2$ to a shorter $(k-1)$-pre-cluster with the same marked non-middle column. Let $C$ be a $k$-pre-cluster of length $2k+i-2$ with marked non-middle column $c$. If $c$ is a top column, then $\Phi(C,c)$ is obtained by deleting the bottom row of $C$ while leaving $c$ marked. Because the height of $c$ is at most $k-2$, it cannot be both a top and a bottom column in $C$, and thus $c$ will be unchanged in $\Phi(C,c)$. If $c$ is a bottom column, $\Phi(C,c)$ is obtained similarly by deleting the top row of $C$ while again leaving $c$ marked.

The map is inverted as follows. Let $C'$ be a $(k-1)$-pre-cluster of length at most $2k+i-3$ and with marked non-middle column $c'$ of height at most $k-2$. Again, by the height restriction on $c'$, it cannot be both a top and a bottom column. If $c'$ is a top column, then there is a unique $k$-pre-cluster of length $2k+i-2$ which has $c'$ as a top column, which is obtained by adding a new bottom row to $C'$ in the unique position to obtain a $k$-pre-cluster of length $2k+i-2$. The resulting $k$-pre-cluster $C$ (with marked column $c$ corresponding to $c'$) clearly has the property that $\Phi(C,c) = (C',c')$. Analogously, if $c'$ is a bottom column, then there is exactly one way to build a $k$-pre-cluster $C$ with marked bottom column $c'$ such that $\Phi(C,c) = (C',c')$.
\end{proof}

Our next step is to study the contribution of columns of height $k$ or $k-1$ to \eqref{main-eq}. There are two cases, $i=1$ and $i\ge 2$. We first consider the special case $i=1$. In this case we see that there is a unique $k$-pre-cluster of length $2k-1$:
\[
\fnmatrix{cccccccc}{
u_1&\cdots&u_{k-1}&u_k\\
&\cdots&u_{k-2}&u_{k-1}&u_k\\
&&\vdots&\vdots&\vdots\\
&&u_1&u_2&u_3&\cdots\\
&&&u_1&u_2&u_3&\cdots&u_k
}
\]
The columns of height $k$ or $k-1$ in this pre-cluster correspond to the entries $u_{1,\dots,k}$, $u_{1,\dots,k-1}$ and $u_{2,\dots,k}$ in the associated symbolic cluster. There is also a unique $(k-1)$-pre-cluster of length $2k-2$:
\[
\fnmatrix{cccccccc}{
u_1&\cdots&u_{k-2}&u_{k-1}&u_k\\
&\cdots&u_{k-1}&u_{k-2}&u_{k-1}&u_k\\
&&&\vdots&\vdots&\vdots\\
&&&u_1&u_2&u_3&\cdots&u_k
}
\]
This pre-cluster contains two columns of height $k-1$, corresponding to the entries $u_{1,\dots,k-1}$ and $u_{2,\dots,k}$ in the associated symbolic cluster. These two entries cancel with two of the three entries from the symbolic $k$-cluster. Because all other terms of \eqref{main-eq} cancel by Proposition~\ref{prop-cancel-1} in this case, \eqref{main-eq} reduces to $u_{1,\dots,k}=\lambda_1$ when $i=1$, as desired. Notice that when $i=1$ all columns which contribute to \eqref{main-eq} are top or bottom columns, and thus we are completely done with this case and will not consider it again in our proof.

Our next propositions give the $i\ge 2$ case.

\begin{proposition}
\label{prop-top-and-bottom}
In the expansion of \eqref{main-eq} for $i\ge 2$, the total contribution of top and bottom columns is $(k-1)\lambda_1+\lambda_2$.
\end{proposition}
\begin{proof}
We begin by looking at such columns in $(k-1)$-pre-clusters. Trivially, these pre-clusters cannot contain columns of height $k$. If a $(k-1)$-pre-cluster contains a column of height $k-1$ then it can have length at most $2k-1$. As in the $i=1$ case, there is a unique $(k-1)$-pre-cluster of length $2k-2$ with a column of height $k-1$. This pre-cluster actually has two columns of height $k-1$, corresponding to the terms $u_{1, \ldots, k-1}$ and $u_{2, \ldots, k}$. Additionally, for every other subset $I\subseteq\{1,\dots k\}$ of size $k-1$, there is a unique $(k-1)$-pre-cluster of length $2k-1$ which has a column corresponding to $u_I$.

Next we consider columns of height at least $k-1$ in $k$-pre-clusters. There is a unique $k$-pre-cluster with a column of height $k$, but this pre-cluster has length $2k-1$ so will not contribute to \eqref{main-eq} when $i\ge 2$. Now choose any subset $I\subseteq\{1,\dots,k\}$ of size $k-1$. There are precisely two $k$-pre-clusters with a marked top or bottom column $u_I$ of each length between $2k$ and $3k-3$ (i.e., for $i$ between $2$ and $k-1$), because $u_I$ can arise from either a top column or a bottom column in such pre-clusters.

Therefore the total contribution of these columns is
\[
	\sum_{\mathclap{\substack{I\subseteq\{1,\dots,k\},\\ |I|=k-1}}} u_I
	=
	(k-1)\lambda_1+\lambda_2,
\]
as desired.
\end{proof}

In the $i=2$ case, the terms of \eqref{main-eq} correspond to $k$-pre-clusters of length $2k$ and $(k-1)$-pre-clusters of lengths $2k-1$ and $2k-2$. As no column of such a pre-cluster can be a middle column, our claim follows from Proposition~\ref{prop-top-and-bottom} in this case. We move on to consider the contribution of middle columns when $i \geq 3$.

\begin{proposition}
\label{prop-cancel-2}
For all $i\ge 3$, middle columns of height less than $k-i+1$ do not contribute to the expansion of \eqref{main-eq}.
\end{proposition}
\begin{proof}
Let $C$ be a $k$-pre-cluster of minimal length which contains a middle column $c$ of height $k-i$. Since $c$ is a middle column, $C$ must have rows both above and below it. Together, these two rows and the column $c$ contribute at least $2k+1$ to the length of $C$. Moreover, $C$ has $(k-(k-i+2))=i-2$ more rows. Because each such row increases the length of $C$ by at least $1$, $C$ has total length at least $2k+i-1$. The expression \eqref{main-eq} does not involve any $k$-clusters of this length. Obviously, the same argument holds for shorter middle columns as well (those of length less than $k-i$).

Analogously, any $(k-1)$-pre-cluster which contains a middle column of height at most $k-i$ has length at least $2k+i-2$. Again, the expression \eqref{main-eq} does not involve any $(k-1)$-clusters of this length.
\end{proof}

\begin{proposition}
\label{prop-middles-1}
Fix $3 \leq i \leq k-1$ and a column $c$ of height at least $k-i+1$. There is a bijection between $k$-pre-clusters of length $2k+i-2$ which contain $c$ as a marked middle column and $(k-1)$-pre-clusters of length between $2k-2$ and $2k+i-3$ which contain $c$ as a marked middle column.
%
\end{proposition}
\begin{proof}
Let $h$ denote the height of $c$. To construct our bijection we first consider the $k$-pre-clusters. Let $c$ and $c'$ be columns of the same height $h\ge k-i+1$. We define a bijection, $\Psi_{c}^{c'}$, which maps from a $k$-pre-cluster of length $2k+i-2$ with a marked middle column $c$ to a $k$-pre-cluster of the same length with marked middle column $c'$. Showing that $\Psi_{c}^{c'}$ is indeed a bijection will obviously prove our claim about middle columns in $k$-pre-clusters.

Given such a $k$-pre-cluster $C$ with marked middle column $c$, the map $\Psi_{c}^{c'}(C,c)$ is defined by replacing the rows which involve $c$ in $C$ by the rows which involve $c'$ (which are the same in any pre-cluster it is a column of) in such a way that $c'$ (which becomes our new marked column) is in the same column of $\Psi_{c}^{c'}(C,c)$ as $c$ was in $C$. Note that, because $c$ was a middle column, the length of $\Psi_{c}^{c'}(C,c)$ will be the same as the length of $C$.

However, we need to show that $\Psi_{c}^{c'}(C,c)$ is indeed a pre-cluster. The concern we need to address is that it is a priori possible for the rows of $\Psi_{c}^{c'}(C,c)$ to fail to overlap. As we will show, however, this is prevented by our length condition. Consider, for the moment, only the rows which involve $c'$. The least entry that can lie on the top of the column $c'$ is $u_h$, and thus this set of rows contains at least $h-1$ columns to the left of $c'$. Similarly, the greatest entry than can lie on the bottom of $c'$ is $u_{k-h+1}$, so this set of rows contains at least $k-(k-h+1)=h-1$ columns to the right of $c'$. Were the rows of $\Psi_{c}^{c'}(C,c)$ to fail to overlap, there would be one row completely to the left (and above) the rest of the array, or a row completely to the right (and below) the array. Let us suppose for the sake of contradiction that there is a row of $\Psi_{c}^{c'}(C,c)$ completely to the left of the rest of the array. This row therefore contributes its total length, $k$, to the length of $\Psi_{c}^{c'}(C,c)$. Moreover, there must be at least $h-1$ columns separating it from the column $c'$ in $\Psi_{c}^{c'}(C,c)$. Thus the column $c'$ contributes $1$ to the length of $\Psi_{c}^{c'}(C,c)$, and because $c'$ is a middle column in $\Psi_{c}^{c'}(C,c)$, there must be a row which contributes its total length to the length of $\Psi_{c}^{c'}(C,c)$. Finally, we have only accounted for $h+2$ rows, and each additional row must contribute at least $1$ to the length of $\Psi_{c}^{c'}(C,c)$. This shows that the length of $\Psi_{c}^{c'}(C,c)$ is at least
$$
k+h-1+1+k+(k-(h+2))=3k-2.
$$
However, the length of $\Psi_{c}^{c'}(C,c)$ is $2k+i-2$ (the length of $C$), which is at most $3k-3$ because $i\le k-1$.

This contradiction completes the proof of the proposition for $k$-pre-clusters. The case of $(k-1)$-pre-clusters is completely analogous.
\end{proof}

Our claim about \eqref{main-eq} in the $i \geq 3$ case, and thus also our proof of Theorem~\ref{thm-rearrangement}, follows from the following result.

\begin{proposition}
For $3\le i\le k-1$, the contribution of middle columns to the expansion of \eqref{main-eq} is a linear combination of $\lambda_1,\dots,\lambda_i$ with strictly positive coefficients.
\end{proposition}
\begin{proof}
Consider the contribution of a particular middle column $c$ to the expansion of \eqref{main-eq}. The $k$-pre-clusters of length $2k+i-2$ with this middle column will contribute positively, while the $(k-1)$-pre-clusters of lengths between $2k-2$ and $2k+i-3$ contribute negatively.

Suppose that $(C,c)$ is a $(k-1)$-pre-cluster of length between $2k-2$ and $2k+i-3$ with marked middle column $c$. Because $c$ is a middle column of $C$, $C$ must contain at least $k$ columns both to the left and the right of $c$. Therefore by adding a new bottom column in the appropriate position we can create a $k$-pre-cluster of length $2k+i-2$ containing marked middle column $c$. This mapping defines an injection from $(k-1)$-pre-clusters of lengths between $2k-2$ and $2k+i-3$ and $k$-pre-clusters of length $2k+i-2$, both with marked middle column $c$. The corresponding terms of \eqref{main-eq} therefore cancel.

However, there are also terms of \eqref{main-eq} corresponding to middle columns of $k$-pre-clusters where the middle column begins in the second row from the bottom of the $k$-pre-cluster, and these terms are not canceled in \eqref{main-eq}. Therefore each middle column $c$ of height between $k-i+1$ and $k-2$ contributes positively to \eqref{main-eq}, and this contribution is the same for each column $c$ by the previous proposition. Finally, note that the contribution of one copy of every column of height $h$ for $k-i+1\le h\le k-2$ is
\[
	\sum_{\mathclap{\substack{I\subseteq\{1,\dots,k\},\\ |I|=h}}} u_I
	=
	{k-1\choose h-1}\lambda_1+{k-2\choose h-1}\lambda_2+\cdots+{h-1 \choose h-1}\lambda_{k-h+1},
\]
and thus the contribution of these columns to \eqref{main-eq} is a linear combination of $\lambda_1,\dots,\lambda_i$ with positive coefficients, completing the proof.
\end{proof}

\section{Further applications of the cluster method}

In our final section, we provide short new proofs using the cluster method for a result and a two conjectures of Kitaev, Liese, Remmel, and Sagan~\cite{kitaev:rationality-irr:}. For our first result we need to introduce some new notation. Given $u\in\P^\ast$, we write $1u$ for the word obtained by prepending the letter $1$ to $u$, and we write $u^+$ for the word obtained by adding $1$ to every letter of $u$.

\begin{proposition}[\protect{\cite[Lemma~4.1]{kitaev:rationality-irr:}}]
\label{prop-we-1}
For $u,v\in\P^\ast$, we have the following Wilf-equivalences,
\begin{enumerate}
\item[(a)] $u$ is Wilf-equivalent to its reverse, $u^\text{r}$,
\item[(b)] $u$ and $v$ are Wilf-equivalent if and only if $1u$ and $1v$ are Wilf-equivalent,
\item[(c)] if $u$ and $v$ are Wilf-equivalent then $u^+$ and $v^+$ are Wilf-equivalent.
\end{enumerate}
In fact, the proposition also holds with Wilf-equivalence replaced by strong Wilf-equivalence.
\end{proposition}
\begin{proof}
For part (a), note that the reversal of a minimal $m$-cluster of $u$ is a minimal $m$ cluster of $u^\text{r}$ with the same length and sum of its entries. Thus $M_u(x,y,z)=M_{u^\text{r}}(x,y,z)$, so $u$ is strongly Wilf-equivalent to its reverse (so the two words are trivially Wilf-equivalent).

Next, we consider how to turn minimal $m$-clusters of $u$ into minimal $m$-clusters of $1u$. Suppose that we have $d$ minimal clusters of $u$, $c^{(1)},\dots,c^{(d)}$, where $c^{(i)}$ is an $m_i$-cluster of $u$. Then the concatenation $1c^{(1)}\cdots c^{(d)}$ is a minimal $(m_1+\cdots+m_d)$-cluster of $1u$. Moreover, because every minimal cluster of $1u$ is of this form, we see that
\[
	M_{1u}(x,y,z)=\frac{xyM_u(x,y,z)}{1-M_u(x,y,z)}.
\]
Therefore $u$ and $v$ are strongly Wilf-equivalent (resp., Wilf-equivalent) if and only if $1u$ and $1v$ are strongly Wilf-equivalent (resp., Wilf-equivalent).

Finally, part (c) follows immediately from the observation that if $c$ is a minimal $m$-cluster of $u$ then $c^+$ is a minimal $m$-cluster of $u^+$, so $M_{u^+}(x,y,z)=M_{u}(xy,y,z)$.
\end{proof}

Next we establish a conjecture of \cite{kitaev:rationality-irr:} proving the converse of Proposition~\ref{prop-we-1} (c).

\begin{proposition}[\protect{\cite[Item (3) of Subsection 8.4]{kitaev:rationality-irr:}}]
If $u^+$ and $v^+$ are Wilf-equivalent (resp., strongly Wilf-equivalent) then $u$ and $v$ are Wilf-equivalent (resp., strongly Wilf-equivalent).
\end{proposition}
\begin{proof}
From the proof of part (c) of the previous proposition, we see that
\[
	M_u(x,y,z)=M_{u^+}\left(\frac{x}{y},y,z\right).
\]
Therefore if $M_{u^+}(x,y,-1)=M_{v^+}(x,y,-1)$ then $M_u(x,y,-1)=M_v(x,y,-1)$, and of course the same holds in the context of strong Wilf-equivalence (when $-1$ is not substituted for $z$).
\end{proof}

In order to explain many of the Wilf-equivalences that were found in \cite{kitaev:rationality-irr:}, the authors made the following conjecture.

\begin{theorem}[\protect{\cite[Conjecture 8.3]{kitaev:rationality-irr:}}]
\label{thm-we-3}
The words $a1b2c$ and $a2b1c$ are (strongly) Wilf-equivalent for any choice of positive integers $a,b,c\ge 2$.
\end{theorem}
\begin{proof}
We prove the result by constructing a bijection $\Pi$ between $m$-pre-clusters of $a1b2c$ and $m$-pre-clusters of $c1b2a$, which preserves the length and sum of entries of the corresponding minimal $m$-clusters. This will prove the result because it shows that
\[
	M_{a1b2c}(x,y,z)=M_{c1b2a}(x,y,z),
\]
and by Proposition~\ref{prop-we-1} (a), $c1b2a$ is strongly Wilf-equivalent to $a2b1c$.

Let $C$ be an $m$-pre-cluster of $a1b2c$. To construct $\Pi(C)$ we simply replace every row of $C$ by the word $c1b2a$ (without moving the rows). For example, we have the mapping of $3$-pre-clusters
\[
	\begin{tabular}{ccc}
		\fnmatrix{cccccccc}{
		a&1&b&2&c\\
		&&a&1&b&2&c\\
		&&&a&1&b&2&c\\
		\hline
		a&1&\max\{a,b\}&a&\max\{b,c\}&b&c&c
		}
	&$\longrightarrow$&
		\fnmatrix{cccccccc}{
		c&1&b&2&a\\
		&&c&1&b&2&a\\
		&&&c&1&b&2&a\\
		\hline
		c&1&\max\{b,c\}&c&\max\{a,b\}&b&a&a
		}
	\end{tabular}
\]
under $\Pi$. Clearly $\Pi$ is a bijection, but we must show that the cluster corresponding to $\Pi(C)$ has the same length and sum of its entries as the cluster corresponding to $C$. Instead, we prove the stronger claim that these two clusters are rearrangements of each other.

Since $\Pi$ essentially swaps the locations of each $a$ with the locations of each $c$, it is clear that any entry which does not involve $a$ or $c$ or which involves both $a$ and $c$ occurs equally frequently in the clusters corresponding to $C$ and $\Pi(C)$. To complete the proof, we show that in any cluster of $a1b2c$, the entries $a$ and $c$ occur equally often and the entries $\max\{a,b\}$ and $\max\{b,c\}$ occur equally often, from which it follows immediately that the clusters corresponding to $C$ and $\Pi(C)$ are indeed rearrangements of each other.

We prove this by induction, for which the base case (a $1$-pre-cluster of $a1b2c$) is trivial. Let $C$ be an $a1b2c$ pre-cluster such that the cluster corresponding to $C$ has an equal number of $a$ and $c$ entries and an equal number of $\max\{a,b\}$ and $\max\{b,c\}$ entries. Consider the effect of adding another row to the bottom of $C$. There are fifteen possible alignments of the columns involving this new row; in each case the new cluster resulting from adding a row to $C$ preserves equality of the number of occurrences of $a$ and $c$ entries and of $\max\{a,b\}$ and $\max\{b,c\}$ entries. The fifteen total cases (in which we only show the rightmost five columns) are shown below.
	\begin{center}
	\begin{tabular}{|l|l|l|l|l|}\hline&&&&\\[-8pt]
		\fnmatrix{ccccc}{
		c\\
		a&1&b&2&c
		}
	&\fnmatrix{ccccc}{
		2&c\\
		a&1&b&2&c
		}
	&\fnmatrix{ccccc}{
		b&2&c\\
		a&1&b&2&c
		}
	&\fnmatrix{ccccc}{
		1&b&2&c\\
		a&1&b&2&c
		}
	&\fnmatrix{ccccc}{
		c\\
		2&c\\
		a&1&b&2&c
		}\\&&&&\\[-8pt]\hline &&&&\\[-6pt]
	\fnmatrix{ccccc}{
		c\\
		b&2&c\\
		a&1&b&2&c
		}
	&\fnmatrix{ccccc}{
		c\\
		1&b&2&c\\
		a&1&b&2&c
		}
	&\fnmatrix{ccccc}{
		2&c\\
		b&2&c\\
		a&1&b&2&c
		}
	&\fnmatrix{ccccc}{
		2&c\\
		1&b&2&c\\
		a&1&b&2&c
		}
	&\fnmatrix{ccccc}{
		b&2&c\\
		1&b&2&c\\
		a&1&b&2&c
		}\\&&&&\\[-8pt]\hline &&&&\\[-8pt]
	\fnmatrix{ccccc}{
		c\\
		2&c\\
		b&2&c\\
		a&1&b&2&c
		}
	&\fnmatrix{ccccc}{
		c\\
		2&c\\
		1&b&2&c\\
		a&1&b&2&c
		}
	&\fnmatrix{ccccc}{
		c\\
		b&2&c\\
		1&b&2&c\\
		a&1&b&2&c
		}
	&\fnmatrix{ccccc}{
		2&c\\
		b&2&c\\
		1&b&2&c\\
		a&1&b&2&c
		}
	&\fnmatrix{ccccc}{
		c\\
		2&c\\
		b&2&c\\
		1&b&2&c\\
		a&1&b&2&c
		}\\[-8pt]&&&&\\\hline
	\end{tabular}
	\end{center}

For example, in the pre-cluster in the top-right corner, adding the row $a1b2c$ destroys one $c$ entry from the corresponding cluster by turning it into a $\max\{a,c\}$ entry, preserves one $c$ entry, and creates a new $c$ entry. It does not create or destroy any occurrences of $a$, $\max\{a,b\}$, or $\max\{b,c\}$.

In the pre-cluster in the bottom-right corner, adding the last row destroys an occurrence of $c$, preserves an occurrence of $c$, and creates an occurrence of $c$. Similarly, it both creates and destroys one occurrence of $\max\{b,c\}$, while preserving one other existing occurrence of $\max\{b,c\}$.

Since we have shown that the clusters corresponding to $C$ and $\Pi(C)$ are rearrangements of each other, this completes the proof.
\end{proof}

Finally, note that our proof of Theorem~\ref{thm-we-3} extends to prove a more general result. For positive integers $a,b,c,x,y$ with $a,b,c\ge x,y$, the same proof shows that $axbyc$ is Wilf-equivalent to $aybxc$.

\bibliographystyle{abbrv}
\bibliography{../../../../refs}

\end{document}